\documentclass[a4paper,11pt]{article} \usepackage[english]{babel}
\usepackage{latexsym,color,natbib,xspace,hyperref}
\usepackage{natbib,graphicx}
\newcommand{\ie}{{\em i.e.\/}\xspace}  \newcommand{\etc}{{\em etc.\/}\xspace}
 \newcommand{\bX}{\mathbf X}
 \newcommand{\bZ}{\mathbf Z}
\newcommand{\bx}{\mathbf x} 
\newcommand{\bz}{\mathbf z} \newcommand{\bo}{\mathbf 0}

    \newcommand{\sgn}{\mbox{\rm sign}}
 \newcommand{\pr}{\mbox{\rm Pr}}
\newcommand{\cd}{\,|\,} \newcommand{\half}{\frac{1}{2}}

\newtheorem{expl}{Example}
\newenvironment{ex}{\begin{expl}\rm}{\halm\end{expl}}

\newtheorem{theorem}{Theorem}
\newtheorem{lemma}
{Lemma}
\newtheorem{prop}[theorem]{Proposition}
\newtheorem{cor}[theorem]{Corollary}

\newcommand{\halm}{\hspace*{\fill} $\Box$\par}
\newenvironment{proof}{\noindent {\bf
    Proof. }}{\halm\vspace{\baselineskip}}

\newcommand{\eqref}[1]{\mbox{(\ref{eq:#1})}}

\newcommand{\lemref}[1]{\mbox{Lemma~\ref{lem:#1}}}

\newcommand{\tabref}[1]{\mbox{Table~\ref{tab:#1}}}

\newcommand{\itref}[1]{\mbox{\ref{it:#1}}}

\newcommand{\propref}[1]{\mbox{Proposition~\ref{prop:#1}}}

\newcommand{\corref}[1]{\mbox{Corollary~\ref{cor:#1}}}

\newcommand{\exref}[1]{\mbox{Example~\ref{ex:#1}}}

\newcommand{\E}{{\mbox{\rm E}}} \newcommand{\V}{{\mbox{\rm var}}}

\newcommand{\cip}{\mbox{$\perp\!\!\!\perp$}}

\begin{document}
\title{{\sc A Note on Prediction Markets }} \author{ A. Philip
  Dawid\thanks {Leverhulme Emeritus Fellow, University of Cambridge,
    UK.}  \and Julia Mortera\thanks {Universit\`a Roma Tre, Italy.  }
} \date{\today}
\maketitle

\begin{abstract}
  \noindent In a prediction market, individuals can sequentially place
  bets on the outcome of a future event.  This leaves a trail of
  personal probabilities for the event, each being conditional on the
  current individual's private background knowledge and on the
  previously announced probabilities of other individuals, which give
  partial information about their private knowledge.  By means of
  theory and examples, we revisit some results in this area.  In
  particular, we consider the case of two individuals, who start with
  the same overall probability distribution but different private
  information, and then take turns in updating their probabilities.
  We note convergence of the announced probabilities to a limiting
  value, which may or may not be the same as that based on pooling
  their private information.\\[1ex]


  \noindent {\small {\em Some key words:} consensus; expert opinion;
    information aggregation; probability forecast; sequential
    prediction}

\end{abstract}
\section{Introduction}
\label{sec:intro}
This paper revisits some of the results appearing in economic
literature from a statistical point of view.

A prediction market -- also known as a predictive market, an
information market, a decision market, or a virtual market -- is a
venue where actors trade predictions on uncertain future events and
can also allow participants to stake bets on the likelihood of various
events occurring. These events include, for example, an election
result, a terrorist attack, a natural disaster, commodity prices,
quarterly sales or even sporting outcomes.  Prediction markets also
offer trade in possible future outcomes on securities markets, in
which case participants who use it are buying something like a futures
contract.  The Iowa Electronic Markets
(\url{http://tippie.uiowa.edu/iem/}) of the University of Iowa Henry
B. Tippie College of Business is one of the main prediction markets in
operation. Also companies like Google have setup their own internal
prediction markets.  Prediction markets sometimes operate on an open
market like the stock market, or in a closed market akin to a betting
pool. A prediction market translates the wisdom of crowds into
predictive probabilities.

For example, suppose that in a prediction market one can bet whether
$A$ occurs (before time $t$), where actors buy and sell contracts
among each other. Let a contract pay 1 if event $A$ occurs and 0
otherwise. Say, the market price for the contract is 0.58.  Offers to
buy and sell are fixed at 0.57 and at 0.59, respectively. Now, you can
either pay 0.59 instantly, or post an offer to pay 0.58 and see if any
actor is willing to sell at that price. Now, the current market price,
0.58, is the consensus probability.

Prediction markets have been studied by \citet{aldous13, arrow08,
  hanson03, chen10, hanson06, wolfers08, strahl15, ehm_etal15}, among
others.

\section{Basic setup}
\label{sec:basic}

We shall focus on the opinions of a specific individual, ``You'', and
how these opinions change in the light of accumulating experience.
Taking a fully Bayesian position, we suppose that Your opinions are
expressed as a joint probability distribution, $\pr$, over all
relevant variables.  Other individuals may have their own
probabilities for various events, but for You these are treated simply
as potential data.  In the sequel, all probabilities are computed
under Your distribution $\pr$.

We shall interpret the term ``expert'' in the sense of
\citet{degroot88,ddm95}.  That is, an individual $E$ is an expert (for
You) if $E$ started with exactly the same joint probability
distribution $\pr$ over all relevant variables as You, and has
observed everything that you have observed, and possibly more.  If You
now learn (just) the probability $\Pi$ that $E$ assigns to some event
$A$, your updated probability for $A$ will be $\Pi$.  That is, You
will agree with the expert.

In the context of a prediction market, experts $E_1, E_2, \ldots$,
sequentially give their probability predictions
$\Pi_1, \Pi_2, \ldots$, for an uncertain event $A$. $E_i$ is the expert 
that makes the forecast at time $i$. We allow for the case that it could be the same
 expert giving his 
forecast at different times.  At time $i$
expert $E_i$ has access to all previous forecasts
$\Pi_1,\ldots, \Pi_{i-1}$, and additional private information $H_i$.
$E_i$ will typically not have access to the private information sets
$H_1, \ldots, H_{i-1}$ that the previous experts used in formulating
their forecasts, but only to the actual forecasts made.  However, in
some markets there is an option for forecasters to leave comments,
which could give additional partial information $K_i$ (which might be
empty) about $H_i$.  We assume that each forecaster is aware of all
such past comments.  Thus $\Pi_i= \pr(A \cd T_i)$, where
$T_i := (K_1, \Pi_1, \ldots, K_{i-1}, \Pi_{i-1}, H_i)$ is the total
information available to $E_i$.

The full public information available just after time $i$ is
$S_i := (K_1, \Pi_1, K_2, \Pi_2, \ldots,K_i, \Pi_i)$.  Note that $S_i$
and $T_i$ both contain all the information made public up to time
$i-1$.  They differ only in the information they contain for time $i$:
here $T_i$ specifies the totality, $H_i$, of expert $E_i$'s
information, both public, $K_i$, and private, whereas $S_i$ specifies
only $E_i$'s public information, $K_i$, and her announced probability
forecast, $\Pi_i$, for $A$ at time $i$.  The information sets $(T_i)$
are not in general increasing with $i$, since $H_i$ is included in
$T_i$ but need not be in $T_{i+1}$.  The information sets $(S_i)$ are
however increasing.  The following Lemma and Corollary show that, for
You, for the purposes of predicting $A$ both information sets $T_i$
and $S_i$ are equivalent, and Your associated prediction is just the
most recently announced probability forecast.

\begin{lemma}
  \label{lem:1}
  $\pr(A \cd S_i) = \pr(A \cd T_i) = \Pi_i$.
\end{lemma}

\begin{proof}
  Since $T_i \supset S_i \ni \Pi_i$,
  \begin{eqnarray*}
    \pr(A \cd S_i) &=& \E\{\pr(A \cd T_i) \cd S_i\}\\
                   &=& \E(\Pi_i \cd S_i)\\
                   &=& \Pi_i\\
                   &=& \pr(A \cd T_i).
  \end{eqnarray*}
\end{proof}

\begin{cor}
  If You observe the full public information $S_i$, and have no
  further private information, Your conditional probability for $A$ is
  just the last announced forecast $\Pi_i$.
\end{cor}


\section{Convergence}
\label{sec:conv}
From \lemref{1} and the fact that the information sequence $(S_i)$ is
increasing, we have:

\begin{cor}
  \label{cor:mart}
  The sequence $(\Pi_i)$ is a martingale with respect to $(S_i)$.
\end{cor}

Then by \corref{mart} and the martingale convergence theorem, we now
have:
\begin{cor}
  \label{cor:limit}
  As $i \rightarrow \infty$, $\Pi_i$ tends to a limiting value
  $\Pi_\infty$.
\end{cor}

The variable $\Pi_\infty$ is random in the sense that it depends on
the initially unknown (to You) information sequence
$S_\infty := \lim S_i$ that will materialise, but will be a fixed
value for any such sequence.

A perhaps surprising implication of \corref{limit} is that,
eventually, introduction of new experts will not appreciably change
the probability You assign to $A$ --- whatever new private information
they may bring will be asymptotically negligible compared with the
accumulated public information.  We may term $\Pi_\infty$ the {\em
  consensus probability\/} of $A$, and the information $S_\infty$ on
which it is based the {\em consensus information set\/}.

The information $S_\infty$ is {\em common knowledge} for all experts
 in the sense of \citet{aumann76}. For details See
\citet{geanakoplos:TARP,geanakoplos:TEP,nielsen, mckelvey}.

It might be considered that the limiting value $\Pi_\infty$ has
succeeded in integrating all the private knowledge of the infinite
sequence of experts.  As we shall see below this is sometimes, but not
always, the case.



\section{Two experts}
\label{sec:2exp}

As a special case, suppose we have a finite set of experts,
$E_1, \ldots,E_N$, and we take $E_{N+1} = E_1$ (so $H_{N+1} = H_1$),
$E_{N+2} = E_2$, \etc.  Thus we repeatedly cycle through the experts.
Continuing for many such cycles, eventually we will get convergence,
to some $\Pi_\infty$ --- at which point each expert will not be
changing her opinion based on the total sequence of publicly announced
forecasts, even though she still has access to additional private
information.

At convergence, it will thus make no difference to Expert~$E_i$ to
incorporate (again) her private information $H_i$.  Consequently we
have:
\begin{prop}
  \label{prop:ind}
  For each $i$, $A \,\cip\, H_i \mid S_\infty$.
\end{prop}

\citet{dutta14} give a simple example with two experts, that shows
that the order in which the experts play can matter.  In their example
they show that when one of the experts starts playing they reach a
complete consensus, whereas changing the order in which they play they
reach a limited consensus.  \citet{dutta14} also show that if an
expert has additional information this can lead to a weaker
consensus. They call this ``obfuscation''.

In the sequel we consider in detail the case $N=2$ of two experts, who
alternate $E_1, E_2, E_1, E_2, \ldots$ in updating and announcing
their forecasts.
\citet{greeks} have studied this in the case that there is no
side-information, and each expert $E_i$'s set of possible private
information has finite cardinality, $K_i$ say.  They show that exact
consensus is reached in at most $K_1+K_2$ rounds.

\subsection{Vacuous consensus}
\label{counter}

We start with some examples where the experts learn nothing from each
other's forecasts---although they would learn more if they were able
to communicate and pool their private data.

\begin{ex}
  \textbf{Parity check}\\
  \label{ex:parity}
  This example is essentially the same as that described by
  \citet[p.~198]{greeks}.

  Let $X_1, X_2$ be independent fair coin tosses.  Expert $E_i$
  observes only $X_i$ ($i=1,2$).  Let $A$ be the event $X_1 = X_2$.
  This has prior probability $0.5$.

  On observing his private information $X_1$, whatever value it may
  take, $E_1$'s probability of $A$ is unchanged, at $0.5$.  His
  announcement of that value is therefore totally uninformative about
  the value of $X_1$.  Consequently $E_2$ can only condition on her
  private information about $X_2$---which similarly has no effect.
  The sequence of forecasts will thus be $0.5, 0.5, 0.5, \ldots$.
  Convergence is immediate, but to a vacuous state.

  However, if the experts could pool their data, they would learn the
  value of $A$ with certainty.
\end{ex}

%
%

\begin{ex}
  \textbf{Bivariate normal}\\
  With this example, we generalise from predicting an uncertain event
  to predicting an uncertain quantity.

  Suppose that $E_1$ observes $X_1$, and $E_2$ observes $X_2$, where
  $(X_1, X_2)$ have a bivariate normal distribution with means
  $\E(X_i)=0$, variances $\V(X_i)=1$, and unknown correlation
  coefficient $\rho$---which is what they have to forecast.  Let
  $\rho$ have a prior distribution $\Pi_0$.  Since $X_1$ is totally
  uninformative about $\rho$, $E_1$'s first forecast is again $\Pi_0$,
  and so is itself uninformative.  Again, $E_2$ has learned nothing
  relevant to $\rho$, and so outputs forecast $\Pi_0$; and so on,
  leading to immediate convergence to a vacuous state.  However the
  pooled data $(X_1,X_2)$ is informative about $\rho$ (though does not
  determine $\rho$ with certainty).
\end{ex}

In the above examples, each expert's private information was
marginally independent of the event or variable, generically $Y$ say,
being forecast, with the immediate result that the consensus forecast
was vacuous, the same as the prior forecast.  Conversely, suppose the
consensus is vacuous.  That is to say,
\begin{equation}
  \label{eq:vac}
  Y \,\cip\, S_\infty.
\end{equation}
But from \propref{ind} (trivially generalised) we have
\begin{equation}
  \label{eq:ind1}
  Y \,\cip\, H_i \mid S_\infty.  
\end{equation}
Combining \eqref{vac} and \eqref{ind1}, we obtain
\begin{math}
  Y \,\cip\, (H_i, S_\infty)
\end{math}
whence, in particular,
\begin{displaymath}
  Y \,\cip\, H_i.
\end{displaymath}
Hence the consensus will be vacuous if and only if each expert's
private information is, marginally, totally uninformative.  The
argument extends trivially to any finite number of experts.

\subsection{Complete consensus}
We use the term \textit{complete consensus} to refer to the case that
the consensus forecast will be the same as the forecast based on the
totality of the private information available to all the individual
forecasters.  A simple situation where this will occur is when $\Pi_i$
is a one-to-one function of $H_i$, so that, by announcing $\Pi_i$,
expert $E_i$ fully reveals her private information.

\begin{ex}  \textbf{Overlapping Bernoulli trials}\\
  \label{ex:bin}
  Let $\theta$ be a random variable with a distribution over $[0,1]$
  having full support.  Given $\theta$, let $Y_0 \sim B(n_0, \theta)$,
  $Y_1 \sim B(n_1, \theta)$, $Y_2 \sim B(n_2, \theta)$, and
  $A \sim B(1, \theta)$, all independently.

  Suppose $E_1$ observes $X_1=Y_0+Y_1$, and $E_2$ observes
  $X_2=Y_0+Y_2$.  At the first stage, $E_1$ computes and announces
  $\Pi_1= \pr(A \cd X_1)$---which is a one-to-one function of $X_1$.
  For example, under a uniform prior distribution for $\theta$,
  $\Pi_1= (X_1+1)/(n_0+n_1+2)$.  Then at stage~2, $E_2$ will have
  learned $X_1$, and also has private information $X_2$.  Thus
  $\Pi_2 = \Pr(A \mid X_1, X_2)$, the correct forecast given the
  complete private information of $E_1$ and $E_2$.  Further cycles
  will not change this probability, which will be the consensus.
\end{ex}

\begin{ex}
  \textbf{Linear prediction}\\
  \label{ex:linear}
  Consider variables $\bX = (X_1, \ldots,X_k)$,
  $\bZ = (Z_1, \ldots, Z_h)$ and (scalar) $Y$, all being jointly
  normally distributed with non-singular dispersion matrix.  Expert~1
  observes $H_1 = \bX$, Expert~2 observes $H_2 = \bZ$, and they have
  to forecast $Y$.  Each time an expert announces her predictive
  distribution for $Y$, she is making known the value of her
  predictive mean of $Y$, which will be some linear combination of the
  predictor variables $(\bX,\bZ)$.  So generically we would expect
  convergence of the forecasts, after at most $\min\{k,h\}$ rounds, to
  the full forecast based on the pooled information $(\bX,\bZ)$.

  In order to investigate this we have made use of the 93CARS dataset
  \citep{lock93}, containing information on new cars for the 1993
  model year.  There are $n=82$ complete cases with information on 26
  variables, including price, mpg ratings, engine size, body size, and
  other features.  We took $\bX = (X_1, \ldots,X_{11})$ to be the
  variables 7 to 17, $\bZ = (Z_1, \ldots, Z_9)$ to be the variables 18
  to 26, and $Y$ to be variable~5 (Midrange Price).

  Let $S$ denote the uncorrected sum-of-squares-and-products matrix
  based on the data for these variables.  The fictitious model we
  shall consider for the prediction game has $(\bX, \bZ, Y)$
  multivariate normal, with mean $\bo$ and dispersion matrix
  $\Sigma = S/n$.  The predictive distribution of $Y$, based on any
  collection of linear transforms of the $X$'s and $Z$'s, will then be
  normal, with a mean-formula that can be computed by running the
  zero-intercept sample linear regression of $Y$ on those variables,
  and variance that will not depend on the values of the predictors.
  Note that, although our calculations are based on the sample data,
  the values computed are not estimates, but are the correct values
  for our fictitious model.

  Let $U_1$ be the variable so obtained from the sample regression $Y$
  on $\bX \equiv (X_1,\ldots,X_{11})$.  Recall that both experts are
  supposed to know the model, hence $\Sigma$, and know which variables
  each is observing.  Consequently both know the form of $U_1$, but
  initially only $E_1$, who knows the values of $(X_1,\ldots,X_{11})$,
  can compute its value, $u_1$ say.  Since his round-1 forecast for
  $Y$ is normal with mean $u_1$, while its variance is already
  computable by both experts, the effect of $E_1$ issuing his forecast
  is to make the value $u_1$ of $U_1$ public knowledge.

  It is now $E_2$'s turn to play.  At this point she knows the values
  of $U_1$ and $(Z_1, \ldots, Z_9)$, and her forecast is thus obtained
  from the sample regression of $Y$ on these variables.  Let this
  regression function (computable by both experts) be $V_1$; then at
  this round $E_2$ effectively makes the value $v_1$ of $V_1$ public.

  Now at round $2$, $E_1$ regresses $Y$ on $(X_1,\ldots,X_{11},V_1)$
  ($U_1$, which is a linear function of his privately known $X$'s,
  being redundant), and announces the value $u_2$ of the computed
  regression function $U_2$.  And so on.

  The relevant computations are easy to conduct using the statistical
  software package {\tt R} \citep{Manual:R}.  At each stage, we
  computed the 82 fitted values based on the regression just
  performed.  These can then be used as values for the new predictor
  variable to be included in the next regression.  Moreover,
  convergence of the forecast sequence will be reflected in
  convergence of these fitted values.  We observe this convergence,
  both for the fitted values and the predicted standard deviations,
  from round~10 onwards: as soon as $E_1$ has access to the values of
  $U_1,\ldots,U_9$, he effectively knows $Z_1,\dots,Z_9$, and his
  forecast becomes the same as that based on the pooled data.  And as
  soon as $E_1$ makes that public, $E_2$ can make the same forecast.

  As a numerical illustration, suppose $E_1$ has observed
  $$\bX = \bx = (16,25,2,1,8,4.6,295,6000,1985,0,20.0),$$ and $E_2$ has
  observed $$\bZ = \bz = ( 5,204,111,74,44,31.0,14,3935,1).$$ Before
  entering the prediction market, $E_1$'s point forecast for $Y$,
  based on his data $\bX = \bx$, is $u_1=40.6163$, and $E_2$'s point
  forecast for $Y$, based on her data $\bZ = \bz$, is $v_0=30.6316$.
  If they could combine their data, the forecast, based on
  $(\bX,\bZ) = (\bx,\bz)$, would be $39.73925$.
 
  On entering the market, the sequence of their predictions is as
  given in \tabref{seq}.
  \begin{table}[btp]
    \centering
    \begin{tabular}{r | llllllllllc}
      $i$: &1&2&3&4&5&6&7&8&9&10&\ldots\\ \hline
      $u_i$: &40.62&39.49&39.34&39.51&39.55&39.54&39.66&39.75&39.73917&39.73925&\ldots\\
      $v_i$:  &38.28&39.40&39.46&39.54&39.56&39.63&39.67&39.74&39.73924&39.73925&\ldots
    \end{tabular}
    \caption{Sequence of market predictions for $Y$}
    \label{tab:seq}
  \end{table}
  We see convergence to the value based on all the data at round~10,
  by which point $E_2$, having publicly announced the values of $9$
  predictor variables, has effectively revealed all her
  $9$-dimensional private information to $E_1$.  The predictions of
  both experts will stay the same thereafter.

  As a second illustration, suppose $E_1$ has observed
  $$\bX = \bx = (22, 30, 1, 0, 4, 3.5, 208, 5700, 2545, 1, 21.1),$$ and $E_2$ has
  observed $$\bZ = \bz = (4, 186, 109, 69, 39, 27.0, 13, 3640, 0).$$
  Before entering the prediction market, $E_1$'s point forecast for
  $Y$ is $27.80968$, and $E_2$'s point forecast is $36.593865$.  Their
  market forecasts converge at round~10 to $31.22983$, the forecast
  based on all the data.

  These two cases illustrate within-sequence convergence, but to a
  random (\ie, data-dependent) limit.
\end{ex}

A similar example for a linear prediction that gives the same
  basic results was shown in \citet{dutta14}. They however do not give
  a numerical illustration.


 %

\subsection{Limited consensus}
\label{sec:limited}

In all the above examples, convergence was either to a vacuous state,
or to a complete consensus based on the totality of the pooled private
information.  As the following example shows, it is also possible to
converge to an intermediate state.

\begin{ex}
  \label{ex:ts}
  Suppose $\theta$ and $X_1$ have independent $N(0,1)$ distributions,
  while, given $(\theta, X_1)$, \mbox{$X_2 \sim N(\theta X_1,1).$}
  Expert $E_1$ observes $H_1 =X_1$, while $E_2$ observes $H_2=X_2$.
  The interest is in predicting $\theta$.  A sufficient statistic for
  $\theta$, based on the combined data $(X_1,X_2)$, is
  $ (X_1 X_2, |X_1|) = (S_1, S_2)$, say.  The posterior distribution
  is
  \begin{displaymath}
    \theta \mid (S_1,S_2) = (s_1,s_2) \sim N\left(\frac{s_1}{1+s_2^2},\frac{1}{1+s_2^2}\right).
  \end{displaymath}

  Straightforward computations deliver the joint density of
  $(X_1,X_2)$, marginalising over $\theta$:
  \begin{equation}
    \label{eq:x1x2}
    f(x_1,x_2) = (2\pi)^{-1} (1+x_1^2)^{-\half} \exp-{\half}\left(x_1^2+\frac{x_2^2}{1+x_1^2}\right).
  \end{equation}
  Because \eqref{x1x2} is unchanged if we change the sign of either or
  both of $x_1$ and $x_2$, we deduce (what may be obvious from the
  symmetry of the whole set-up):
  \begin{prop}
    \label{prop:flip}
    Conditionally on $|X_1|$ and $|X_2|$, $\sgn(X_1)$ and $\sgn(X_2)$
    behave as independent fair coin-flips.
  \end{prop}

  At the first round, $E_1$ declares his posterior for $\theta$, based
  on $X_1$---but, since $X_1 \,\cip\,\theta$ this supplies no
  information at all about $\theta$.  (So we would get the same answer
  if $E_2$ were to go first---the order in which they announce their
  opinions does not matter.)

  Now $E_2$ goes.  Since $X_2 \mid \theta\sim N(0, 1+\theta^2)$, with
  sufficient statistic $|X_2|$, $E_2$ is effectively putting $|X_2|$
  into the public pot.

  At the start of round 2, $E_1$ knows $X_1$ and $|X_2|$.  By
  \propref{flip}, $\sgn(X_2)$ is still equally likely to be $1$ or
  $-1$.  So $E_1$ knows $S_2$, but only knows $|S_1|$---for him, $S_1$
  is either $|S_1|$ or $-|S_1|$, each being equally likely. His
  posterior is thus a 50--50 mixture of the associated posteriors
  $$N\left(\frac{|S_1|}{1+S_2^2},\frac{1}{1+S_2^2}\right)$$
  and
  $$N\left(\frac{-|S_1|}{1+s_2^2},\frac{1}{1+S_2^2}\right).$$
  On $E_1$'s now announcing this mixture posterior, he is effectively
  communicating $(|S_1|,S_2) \equiv (|X_1|\times |X_2|, |X_1|)$.  The
  total information in the public pot is thus now equivalent to
  $(|X_1|, |X_2|)$.
 
  It is now $E_2$'s turn again.  At this point she knows
  $(|X_1|,X_2)$, so $(|S_1|, S_2)$---but still does not know
  $\sgn(S_1)$, which again behaves as a coin-flip.  Her forecast
  distribution is thus exactly the same as $E_1$'s.  So we get
  convergence to the above mixture posterior at round 2.  But this
  limiting forecast is not the same as that based on the pooled data,
  which would be the relevant single component of the mixture.

  Note that, at convergence, the pool of public knowledge is
  $(|X_1|, |X_2|)$.  Since $\theta$ has the identical mixture
  posterior whether conditioned on $(|X_1|, |X_2|)$, on
  $(X_1, |X_2|)$, or on $(|X_1|, X_2)$, we have both
  $\theta\,\cip\, X_1 \mid (|X_1|, |X_2|)$ and
  $\theta\,\cip\, X_2 \mid (|X_1|, |X_2|)$, in accordance with
  \propref{ind}.
\end{ex}

It might appear that the above behaviour is highly dependent on the
symmetry of the problem, but this is not so.  As the following
analysis shows, the same limited consensus behaviour arises on
breaking the symmetry.
\begin{ex}
  Consider the same problem as in \exref{ts} above, with the sole
  modification that the prior distribution of $\theta$ is now
  $N(\mu,1)$, where $\mu$ is non-zero.  The posterior distribution of
  $\theta$, based on the full data $(X_1,X_2)$ or its sufficient
  statistic $(S_1,S_2)$, is now
  \begin{displaymath}
    \theta \mid (S_1,S_2) = (s_1,s_2) \sim \Pi(s_1,s_2) := N\left(\frac{\mu+s_1}{1+s_2^2},\frac{1}{1+s_2^2}\right).
  \end{displaymath}

  \renewcommand{\theenumi}{(\roman{enumi})} The following result is
  immediate.
  \begin{prop}
    \label{prop:mixpost}
    Given only $|S_1| = m_1, S_2 = m_2$, the posterior distribution is
    a mixture:
    \begin{equation}
      \label{eq:mixpost}
      \theta \sim M(m_1, m_2) = \pi(1) \Pi(m_1, m_2) + \pi(-1)\Pi(-m_1, m_2)
    \end{equation}
    where
    \begin{equation}
      \label{eq:pi}
      \pi(j) = P(\sgn(S_1) = j \cd |S_1| = m_1, S_2 = m_2)\quad(j = \pm 1). 
    \end{equation}

  \end{prop}
  \begin{prop}
    \label{prop:flip2}
    Conditionally on $|X_1|$ and $|X_2|$:
    \begin{enumerate}
    \item \label{it:x1} $\sgn(X_1)\,\cip\,\sgn(X_1X_2)$
    \item \label{it:x2} $\sgn(X_2)\,\cip\,\sgn(X_1X_2)$
    \end{enumerate}
  \end{prop}

  \begin{proof}
    \itref{x1} The joint density of $(X_1,X_2)$, marginalising over
    $\theta$, is
    \begin{displaymath}
      f(x_1,x_2) = (2\pi)^{-1} (1+x_1^2)^{-\half} \exp-{\half}\left(x_1^2+\frac{(x_2-\mu x_1)^2}{1+x_1^2}\right).
    \end{displaymath}
    This is unchanged if we change the signs of both $x_1$ and $x_2$.
    Consequently, given $|X_1|=m_1, |X_2|=m_2$,
    $P(X_1 = m_1, X_2 = m_2) = P(X_1 = -m_1, X_2 = -m_2)$, while
    $P(X_1 = m_1, X_2 = -m_2) = P(X_1 = - m_1, X_2 = m_2)$.  But this
    is equivalent to
    \begin{eqnarray*}
      P(\sgn(X_1) = 1, \sgn(X_1X_2) = 1) &=& P(\sgn(X_1) = -1,
                                             \sgn(X_1X_2) = 1)\\
      P(\sgn(X_1) = 1, \sgn(X_1X_2) = -1) &=& P(\sgn(X_1) = -1,
                                              \sgn(X_1X_2) = -1). 
    \end{eqnarray*}
    Thus
    $P(\sgn(X_1) = 1 \cd \sgn(X_1X_2) = 1) = P(\sgn(X_1) = 1 \cd
    \sgn(X_1X_2) = -1) = \half$, which in particular implies
    $\sgn(X_1)\,\cip\,\sgn(X_1X_2)$.

    \itref{x2} We have
    \begin{eqnarray*}
      P(\sgn(X_2) = 1 \cd \sgn(X_1X_2) = 1) &=& P(\sgn(X_1) = 1 \cd \sgn(X_1X_2) = 1)\\
      P(\sgn(X_2) = 1 \cd \sgn(X_1X_2) = -1) &=& P(\sgn(X_1) = -1, \sgn(X_1X_2) = -1)
    \end{eqnarray*}
    So from \itref{x1}, conditional on $|X_1|=m_1, |X_2|=m_2$,
    $P(\sgn(X_2) = 1 \cd \sgn(X_1X_2) = 1) = P(\sgn(X_2) = 1 \cd
    \sgn(X_1X_2) = -1) = \half$ so that, in particular,
    $\sgn(X_2)\,\cip\,\sgn(X_1X_2)$.
  \end{proof}
  
  In the first round, $E_1$ and $E_2$ behave exactly as before, and
  again, at the start of round 2, the public pot contains $|X_2|$.  So
  now $E_1$ knows $X_1$ and $|X_2|$.  In terms of the sufficient
  statistic he knows $(|S_1|, S_2)$, but does not know $\sgn(S_1)$.
  Moreover, by \propref{flip2}\itref{x1}, his additional knowledge of
  $\sgn(X_1)$ contains no relevant further information about
  $\sgn(S_1)$.  Consequently, he will compute and announce the mixture
  posterior $M(|S_1|, S_2)$.  From this it is possible to deduce the
  values of $|S_1|$ and $S_2$.  Hence at this point the public pot
  contains $(|S_1|, S_2)$.

  Now $E_2$ knows $(|S_1|, S_2)$, but is still ignorant of
  $\sgn(S_1)$.  And again, although she has the additional knowledge
  of $\sgn(X_2)$, by \propref{flip2}\itref{x2} this contains no
  relevant further information about $\sgn(S_1)$.  Consequently, $E_2$
  will have the same posterior distribution $M(|S_1|, S_2)$, which
  will be the final (but limited) consensus.

  (Note that an essentially identical analysis will hold with any
  prior distribution for $\theta$.)

\end{ex}

\section{Discussion}
\label{sec:disc}
We have displayed a variety of behaviours for a process where two
experts take it in turns to update their probability of a future
event, conditioning only on the revealed probabilities of the other.
Although there will always be convergence to a limiting value, this
may or may not be the same as what they could achieve if they were
able to pool all their private information.

We have supposed throughout that, although each expert may be unaware
of the private information held by the other, he does at least know
which variables the other expert knows---just not their values.  When
this cannot be assumed there will be much greater freedom to update
one's own probability on the basis of the revealed probability of the
other.  Nevertheless this freedom is restricted.  Some theory relevant
to the case of combining the announced probabilities of a number of
experts, without even knowing the private variables on which these are
based, may be found in \citet{ddm95}.  It would be challenging, but
valuable, to extend this to the present sequential case.

\end{document}